\documentclass[11pt]{amsart}

\usepackage{amsmath, amssymb, amsthm, amsfonts, enumerate, color, comment}

\newcommand{\h}{\mathcal{H}}

\newcommand{\aaa}{\alpha}
\newcommand{\bbb}{\beta}
\newcommand{\ccc}{\gamma}

\usepackage{palatino}
\usepackage{graphicx}

\usepackage{parskip}

\numberwithin{equation}{section}
\theoremstyle{plain}
\newtheorem{Proposition}[equation]{Proposition}
\newtheorem{Corollary}[equation]{Corollary}
\newtheorem*{Corollary*}{Corollary}
\newtheorem{Theorem}[equation]{Theorem}
\newtheorem*{Theorem*}{Theorem}
\newtheorem{Lemma}[equation]{Lemma}
\theoremstyle{definition}

\newtheorem{Example}[equation]{Example}
\newtheorem{Remark}[equation]{Remark}
\newtheorem{Question}[equation]{Question}

\usepackage{enumitem}
\setlist[enumerate]{leftmargin=*}
\setlist[itemize]{leftmargin=*}

\def\C{\mathbb{C}}
\def\R{\mathbb{R}}
\def\D{\mathbb{D}}
\def\T{\mathbb{T}}
\def\N{\mathbb{N}}

\def\K{\mathcal{K}}
\def\B{\mathcal{B}}

\renewcommand{\leq}{\leqslant}
\renewcommand{\geq}{\geqslant}
\renewcommand{\subset}{\subseteq}
\renewcommand{\phi}{\varphi}
\renewcommand{\vec}[1]{{\bf #1}}

\usepackage{xcolor}
\newcommand{\0}{{\color{lightgray}0}}

\author[J. Mashreghi]{Javad Mashreghi}
\address{D\'epartement de math\'ematiques et de statistique, Universit\'e Laval, Qu\'ebec, QC,
Canada, G1K 0A6}
\email{javad.mashreghi@mat.ulaval.ca}

\author[M. Ptak]{Marek Ptak}
\address{Department of Applied Mathematics,
University of Agriculture, ul. Balicka 253c\\ 30-198 Krak\'ow, Poland.}
\email{rmptak@cyf-kr.edu.pl}

\author[W. Ross]{William T. Ross}
	\address{Department of Mathematics and Computer Science, University of Richmond, Richmond, VA 23173, USA}
	\email{wross@richmond.edu}
	

\title[Square root]{The square roots of some classical operators}


\thanks{This work was supported by the NSERC Discovery Grant (Canada) and by the Ministry of Science and
Higher Education of the Republic of Poland.}

\begin{document}

\begin{abstract}
In this paper we give complete descriptions of the set of square roots of certain classical operators, often providing specific formulas. The classical operators included in this discussion are the square of the unilateral shift, the Volterra operator, certain compressed shifts, the unilateral shift plus its adjoint, the Hilbert matrix, and the Ces\`{a}ro operator.
\end{abstract}

\maketitle

\section{Introduction}

If $\h$ is a complex Hilbert space and $A \in \B(\h)$, the bounded linear operators on $\h$, when does $A$ have a square root? By this we mean, does there exist a $B \in \B(\h)$ such that $B^2 = A$? If $A$ has a square root, can we describe 
$\{B \in \B(\h): B^2 = A\},$
the set of all the square roots of $A$?

First let us make the, perhaps unexpected, observation that not every operator has a square root. For example, Halmos showed that the unilateral shift $S f = z f$ on the Hardy space $H^2$ \cite{MR3526203} does not have a square root \cite{MR270173}. Other examples of operators constructed with the shift $S$ and its adjoint $S^{*}$, for example $S \oplus S^{*}$ and $S \otimes S^{*}$, also do not have square roots  \cite{MR870760}. See the papers \cite{MR53391, MR979593, MR3980672, MR88698, MR278097} for some general results concerning square roots of operators.  

Second, many operators have an abundance of square roots. For example, any nilpotent operator of order two is a square root of the zero operator.  Moreover, to highlight their abundance, Lebow proved (see \cite[Prob.~111]{MR675952}) that when $\operatorname{dim} \mathcal{H} = \infty$, the set 
$\{A \in \B(\h): A^2 = 0\}$ is dense in $\mathcal{B}(\mathcal{H})$ in the strong operator topology. 

Much of the work on square roots has focused on the general topic of which operators have square roots and the prevalence of types of square roots ($p$th roots and logarithms) in $\B(\h)$. Previous papers also have results which explore the relationship between the type of square root as related to the type of operator. In this paper, we focus on a collection of some well-known classical operators and proceed to characterize {\em all} of their square roots. The classical operators included in this discussion are the square of the unilateral shift (Theorem \ref{8UHG7TF}), the Volterra operator (Theorem \ref{sdfjsdf122445}), certain compressed shifts (Theorem \ref{bBbsbabbBBB}), the unilateral shift plus its adjoint (Theorem \ref{Hilbertertertry}), the Hilbert matrix (Theorem \ref{HMiiiiI}), and the Ces\`{a}ro operator (Theorem \ref{cesarooooo} and Theorem \ref{9099887YY66T+}). Our work on the Ces\`{a}ro operator answers a question posed in \cite{MR979593} and stems from a question posed by Halmos.

\section{Square roots of $S^2$}

Suppose that $(S f)(z) = z f(z)$ denotes the unilateral shift on the Hardy space $H^2$ \cite{MR0133008}.
In this section we explore the square roots of $S^2$. One square root of $S^2$ is, of course, $S$ itself. Our characterization of  {\em all} of the square roots of $S^2$, requires a few preliminaries.

For $g \in H^2$, let 
$$g_{e}(z) := \tfrac{1}{2}(g(z) + g(-z)) \quad \mbox{and} \quad g_{o}(z) := \tfrac{1}{2} (g(z) - g(-z))$$ and observe that 
$g(z) = g_{e}(z) + g_{o}(z).$
If 
$g(z) = \sum_{n = 0}^{\infty} a_n z^n,$
then $$g_{e}(z) = \sum_{k = 0}^{\infty} a_{2 k} z^{2 k} \quad \mbox{and} \quad  g_{o}(z) = \sum_{k = 0}^{\infty} a_{2 k + 1} z^{2 k + 1}.$$
This is the ``even'' and ``odd'' decomposition of $g$ since $g_e(-z) = g_e(z)$ and $g_o(-z) = - g_o(z)$. 
Finally, let 
$$(W g)(z) = \begin{bmatrix}{\displaystyle \sum_{k = 0}^{\infty} a_{2 k} z^k}\\ {\displaystyle \sum_{k = 0}^{\infty} a_{2 k + 1} z^{k}}\end{bmatrix}$$ and note that $W$  is a unitary operator from $H^2$ onto 
$$H^2 \oplus H^2 = \Big\{\begin{bmatrix} f\\g\end{bmatrix}: f, g \in H^2\Big\},$$
with 
$$W^{*} \begin{bmatrix} g_1\\g_2\end{bmatrix} =  g_1(z^2) + z g_{2}(z^2).$$

Our last bit of notation is the vector-valued shift 
$$S \oplus S: H^2 \oplus H^2 \to H^2 \oplus H^2, \quad (S \oplus S) \begin{bmatrix}  g_1\\ g_2\end{bmatrix} =   \begin{bmatrix}  Sg_1\\ S g_2\end{bmatrix}.$$  It is traditional to think of $S \oplus S$ in matrix form as 
$$S \oplus S = \begin{bmatrix} S & 0\\ 0 & S \end{bmatrix}.$$

The above formulas yield the following well-known fact.

\begin{Proposition}\label{554188uU}
$W^{*} (S \oplus S) W = S^2$.
\end{Proposition}

Some other well-known facts used in the section involve the commutants of $S$ and $S \oplus S$. For $\varphi \in H^{\infty}$, the bounded analytic functions on $\mathbb D$, the Toeplitz (Laurent) operator 
$T_{\varphi} g = \varphi g $ is bounded on $H^2$ and $S T_{\varphi} = T_{\varphi} S$. Let 
$$\{S\}' = \{A \in \mathcal{B}(H^2): S A = A S\}$$ denote the commutant of $S$. The following fact is standard \cite[Thm.~3.4]{MR2003221}. 

\begin{Proposition}\label{w7iruyghdfs}
$\{S\}' = \{T_{\varphi}: \varphi \in H^{\infty}\}$.
\end{Proposition}

In a similar way, let
$$\Phi = 
\begin{bmatrix}
\Phi_{11} & \Phi_{12}\\
\Phi_{21} & \Phi_{2 2}
\end{bmatrix},
$$
where $\Phi_{i j} \in H^{\infty}$. We use the notation $M_2(H^{\infty})$ for the $2 \times 2$ matrices above with $H^{\infty}$ entries.  Define $T_{\Phi}: H^2 \oplus H^2 \to H^2 \oplus H^2$ by 
$$T_{\Phi} \begin{bmatrix}g_1\\g_2\end{bmatrix} = 
\begin{bmatrix} \Phi_{11} & \Phi_{1 2}\\ \Phi_{21} & \Phi_{22}\end{bmatrix} 
\begin{bmatrix} g_{1} \\ g_{2} \end{bmatrix} = \begin{bmatrix} \Phi_{11} g_1 + \Phi_{12} g_{2}\\ \Phi_{2 1} g_{1} + \Phi_{2 2} g_2\end{bmatrix}.$$ A calculation shows that $(S \oplus S) T_{\Phi} = T_{\Phi} (S \oplus S)$. Similar to the above, we have the following \cite[Cor.~3.20]{MR2003221}. 

\begin{Proposition}\label{bhubhuUJUJ}
$\{S \oplus S\}' = \{T_{\Phi}: \Phi \in M_{2}(H^{\infty})\}$.
\end{Proposition}

Here is the main theorem of this section describing all of the square roots of $S^2$.

\begin{Theorem}\label{8UHG7TF}
For $Q \in \mathcal{B}(H^2)$ the following are equivalent. 
\begin{enumerate}
\item[(i)] $Q^2 = S^2$. 
\item[(ii)] There is a $2 \times 2$ constant unitary matrix $U$ and  functions $a ,b ,c \in H^\infty$ satisfying
  \begin{equation}\label{condi1} z a^2 + b  c =1
  \end{equation}
such that
 \begin{equation}\label{afor}Q=W^{*}U^*
\begin{bmatrix}
 z a &b \\
z c &-z a 
\end{bmatrix}
UW.\end{equation}
\end{enumerate}
\end{Theorem}

Proposition \ref{554188uU} shows that to prove Theorem \ref{8UHG7TF}, it suffices to prove the following. 

\begin{Theorem}\label{bnnbhjbhj777y}
For $A \in \mathcal{B}(H^2 \oplus H^2)$ the following are equivalent. 
\begin{enumerate}
\item[(i)] $A^2 = S \oplus S$. 
\item[(ii)] There is a $2 \times 2$ constant unitary matrix $U$ and  functions $a ,b ,c \in H^\infty$ satisfying 
  \begin{equation}
   z a^2 + b  c =1
  \end{equation}
such that
 \begin{equation}\label{afor1}A= U^*
\begin{bmatrix}
 z a &b \\
z c &-z a 
\end{bmatrix}
U.\end{equation}
\end{enumerate}
\end{Theorem}

A matrix calculation shows that 
\begin{align*}
A^2 & = U^{*} \begin{bmatrix}
 z a &b \\
z c &-z a 
\end{bmatrix} U U^{*} \begin{bmatrix}
 z a &b \\
z c &-z a 
\end{bmatrix} U\\
& = U^{*} \begin{bmatrix}
 z^2 a^2 + z b c & 0 \\
0 & z^2 a^2 + z b c
\end{bmatrix} U\\
& = U^{*} \begin{bmatrix}
 z  & 0 \\
0  & z 
\end{bmatrix} U\\
& = U^{*} (S \oplus S) U\\
& = S \oplus S.
\end{align*}
In the above, we use the fact that any constant matrix commutes with $S \oplus S$. 
Thus every operator of the form \eqref{afor1} is a square root of $S \oplus S$. The rest of this section will be devoted to proving the converse -- and providing some instances of this characterization.

Our proof involves a few more preliminaries. The first is a simple fact about square roots of bounded Hilbert space operators. 

\begin{Lemma}\label{98yuijokl}
If $B \in \mathcal{B}(\mathcal{H})$ and $A^2 = B$, then $A \in \{B\}'$. 
\end{Lemma}

\begin{proof}
Note that $A B = A A^2 = A^2 A = B A$.
\end{proof}

Combining this with the discussion above, we see that if $Q \in \mathcal{B}(H^2)$ with $Q^2 = S^2$, then $W Q W^{*} \in \mathcal{B}(H^2 \oplus H^2)$ with $(W Q W^{*})^2 = S \oplus S$. It follows from Lemma \ref{98yuijokl} and Proposition \ref{bhubhuUJUJ} that 
$$W Q W^{*} = A, \quad A \in M_{2}(H^{\infty}).$$

To identify $A$, let us start with lemmas about $2\times 2$ matrices $M_2(\mathbb{C})$ of complex numbers. 
For $X,Y\in M_{2}(\C)$ let 
$$ [X,Y]^+ : = XY+YX.$$
One can quickly verify the following useful facts about the subspace 
\begin{equation}\label{zero}
\mathcal{S}=\Big\{\left[\begin{array}{cc}\aaa&\phantom{-}\bbb\\
\ccc&-\aaa\end{array}\right]\ \colon \aaa,\bbb,\ccc\in\mathbb{C}\Big\}.
\end{equation}

\begin{Lemma}\label{techLemma}
Let $\alpha, \beta, \gamma, \lambda \in \C$.
\begin{enumerate}
\item[(i)] If
$$X=
\begin{bmatrix}
\alpha & \beta\\
0 & \gamma
\end{bmatrix}
$$ and $X^2 = 0$, then $\alpha = \gamma = 0$, in other words, $X^2 = 0$ if and only if
$$X = 
\begin{bmatrix}
0 & \beta\\
0 & 0
\end{bmatrix}.$$
\item[(ii)]  If 
$$X = \begin{bmatrix}
0 & \beta\\
0 & 0
\end{bmatrix}$$
and $Y \in M_{2}(\mathbb{C})$ with $[X, Y]^{+} = \lambda I_2$ with $\lambda \not = 0$, then $\beta \not = 0$ and 
$$Y = 
\begin{bmatrix}
\alpha & \eta\\
\lambda/\beta & -\alpha
\end{bmatrix},$$
where $\alpha, \eta \in \mathbb{C}$ are arbitrary. 
\item[(iii)] If $X,Y\in\mathcal{S}$, then  $X^2$ and $[X,Y]^+\in \mathbb{C} I_{2}$.
\end{enumerate}
\end{Lemma}

For a sequence $(A_{k})_{k = 0}^{\infty}$, where $A_k \in M_{2}(\C)$ for all $k \geq 0$, consider the formal sum
$$A = \sum_{k = 0}^{\infty} A_k (S \oplus S)^{k}.$$
Each term $A_k (S \oplus S)^k$ belongs to  $\mathcal{B}(H^2 \oplus H^2)$ as does each partial sum of the series above. If we suppose that the series above converges in the strong operator topology, then $A \in \mathcal{B}(H^2 \oplus H^2)$. Suppose $U \in M_{2}(\C)$ is a constant unitary matrix. A simple $2 \times 2$ matrix calculation shows that 
$$U (S \oplus S)^{k} = (S \oplus S)^{k} U \quad \mbox{for all $k \geq 0$.}$$
This yields the important identity 
\begin{equation}\label{VBHJK}
U A U^{*} = \sum_{k = 0}^{\infty} U A_{k} U^{*} (S \oplus S)^{k}.
\end{equation}

\begin{proof}[Proof of Theorem \ref{8UHG7TF}]
We will prove Theorem \ref{bnnbhjbhj777y}. Proposition \ref{554188uU} will then imply Theorem \ref{8UHG7TF}. 

Let $A \in \mathcal{B}(H^2 \oplus H^2)$ with $A^2  = S \oplus S$. Lemma \ref{98yuijokl} and Proposition \ref{bhubhuUJUJ} together show that 
$$A = 
\begin{bmatrix} a & b\\ 
c & d\\
\end{bmatrix},$$
where 
$a $, $b $, $c $, $d \in H^\infty$. Let $a_k, b_k, c_k, d_k$ denote the Taylor coefficients of $a, b, c, d$ respectively and define 
$$A_{k} = 
\begin{bmatrix}
a_k & b_k\\
c_k & d_k
\end{bmatrix} \in M_{2}(\C), \quad k \geq 0.$$
Notice that 
$$A = \sum_{k = 0}^{\infty} A_{k} (S \oplus S)^k.$$
For the matrix $A_0$, Schur's theorem provides us with a unitary matrix $U$ such that $U A_0 U^{*}$ is upper triangular. By \eqref{VBHJK} (and unitary equivalence) we can always assume that $A$ is a square root of $(S \oplus S)$ with $A_0$ being upper triangular.  

Since  $A^2= S \oplus S$ then
\begin{align*}
S \oplus S & = A^2\\
& = \Big(\sum_{k=0}^\infty A_k (S \oplus S)^k\Big)^2\\
& =  \sum_{k=0}^\infty
\Big(\sum_{m=0}^{k} A_m\,A_{k-m}
  \Big) (S \oplus S)^{k}\\
 & = A_0^2 +  [A_0,A_1]^+ (S \oplus S)\\
 & \quad \quad  + \sum_{k=3,k\  \mbox{{\tiny odd}}}^\infty
\Big([A_0,A_k]^+ +\sum_{m = 1}^{\big[\frac k2\big]} [A_m\,,\,A_{k-m}]^+
  \Big) (S \oplus S)^{k}\\
  & \quad  \quad +\sum_{k=2,k\  \mbox{{\tiny even}}}^\infty
\Big([A_0,A_k]^++\sum_{m=1}^{\big[\frac k2\big]-1} [A_m\,,\,A_{k-m}]^+
 +A_{\frac k2}A_{\frac k2} \Big) (S \oplus S)^{k}.
\end{align*}
Comparing the operator coefficients in front of each $(S \oplus S)^{k}$ we have
\begin{align}
 A_0^2&=0,\label{11} \\
   [A_0,A_1]^+&=I_{2} \quad \mbox{($2 \times 2$ identity matrix)},\label{12}\\
  [A_0,A_k]^+&=-\sum_{m=1}^{\big[\frac k2\big]} [A_m\,,\,A_{k-m}]^+,\quad \text{for}\ k\geqslant 2, k\ \text{odd},\label{13}\\
  [A_0,A_k]^+&=-\sum_{m=1}^{\big[\frac k2\big]-1} [A_m\,,\,A_{k-m}]^+
 -A_{\frac k2}A_{\frac k2}\quad \text{for}\ k\geqslant 2, k\ \text{even}.\label{14}
  \end{align}
Now we will inductively find a formula for $A$. 

The matrix $A_0$ is upper triangular. By \eqref{11} and  Lemma \ref{techLemma}, 
$$A_{0} = 
\begin{bmatrix}
0 & b_0\\
0 & 0
\end{bmatrix}.
$$

By \eqref{12} and Lemma \ref{techLemma} we get $b_0\not=0$ and 
$$A_1 = 
\begin{bmatrix}
a_1 & b_1\\
1/\alpha & -a_1
\end{bmatrix}
$$
for arbitrary $a_1,b_1$. 

We will now use induction to prove that $A_k \in \mathcal{S}$. The base cases $A_0, A_{1}$ belong to $\mathcal{S}$. By Lemma \ref{techLemma}, right hand side of \eqref{13} or \eqref{14} are constant multiples of the identity operator $I$ on $H^2 \oplus H^2$. Thus, by Lemma \ref{techLemma},  $A_k \in \mathcal{S}$. 

By the expansion 
$$A = \sum_{k = 0}^{\infty} A_k (S \oplus S)^{k},$$
and the fact that each $A_{k} \in \mathcal{S}$, yields $a, b, c \in H^{\infty}$ such that 
$$A = \begin{bmatrix}
a & b\\
c & -a
\end{bmatrix}$$
with 
$a(0) = 0$, $c(0) = 0$ and $b(0) \not = 0.$
Since 
$$S \oplus S = \begin{bmatrix} z & 0\\ 0 & z\end{bmatrix} = \begin{bmatrix} 
a & b\\
c & -a
\end{bmatrix}^2 = \begin{bmatrix}
a^2 + b c & 0\\
 0 & a^2 + b c
\end{bmatrix},$$
it follows that $a^2 + b c = z$.
 Equivalently, by relabeling $a, b, c$, we can write 
 $$A = 
 \begin{bmatrix}
 z a & b\\
 z c & -za
 \end{bmatrix},$$
 where $a ,b , c \in  H^\infty$ with 
$
z a ^2+ b c =1.
$

The converse was shown earlier. 
\end{proof}

\begin{Remark}
\begin{enumerate}
\item[(i)] Since unitary operators preserve determinants, every square root $A$ of $S \oplus S$ will satisfy $\det A  = -z$. 
\item[(ii)] It follows from Proposition \ref{bhubhuUJUJ} and Proposition \ref{554188uU} that every $B \in \{S^2\}'$ is of the form 
$(B g)(z) = \phi(z) g_{e}(z) + \psi(z) g_{o}(z)$ for some $\phi, \psi \in H^{\infty}$. This is an interesting (and known) fact. 
\item[(iii)] Taking $U = I_2$ (the $2 \times 2$ identity matrix in $M_2(\C)$) in Theorem \ref{8UHG7TF} yields the following class of square roots $Q$ of $S^2$:
\begin{equation}\label{bsgifdgfd}
(Q g)(z)  = \big(z^2 a(z^2) + z^3 c(z^2)\big) g_{e}(z) 
 + \big(b(z^2) - z^3 a(z^2)\big) \frac{g_{o}(z)}{z},
 \end{equation}
where $z a^2 + b c = 1$. Setting $a \equiv 0, b \equiv 1, c \equiv 1$ ,we get 
$$(Q g)(z) = z^3 g_{e}(z) + \frac{g_{o}(z)}{z}.$$
With respect to the standard basis $(z^n)_{n = 0}^{\infty}$ for $H^2$, the operator $Q$ has the matrix representation 
$$[Q] = \begin{bmatrix}
\0 & 1 & \0 & \0 & \0 & \0 & \0 & \cdots\\
\0 & \0 & \0 & \0 & \0 & \0 &  \0 & \cdots\\
\0 & \0 & \0 & 1 & \0 & \0 &  \0 & \cdots\\
1 & \0 & \0 & \0 & \0 & \0 &\0 & \cdots\\
\0 & \0 & \0 & \0 & \0 & 1 & \0 & \cdots\\
\0 & \0 & 1 & \0 & \0 & \0 &\0 & \cdots\\
\0 & \0 & \0 & \0 & \0 & \0 & \0 & \cdots\\
\vdots & \vdots & \vdots & \vdots & \vdots & \vdots & \vdots & \ddots
\end{bmatrix}.$$
\item[(iv)] Taking 
$$U = \begin{bmatrix}
0 & 1\\
1 & 0
\end{bmatrix}
$$ in Theorem \ref{8UHG7TF}, yields another class of square roots $Q$ of $S^2$:
\begin{equation}\label{098789u878}
(Q g)(z) = (z b(z^2) - z^2 a(z^2)) g_{e}(z) + (z^2 a(z^2) + z c(z^2)) g_{o}(z).
\end{equation}
Setting $a \equiv 1$, $b(z) = \sqrt{1 - z}$, $c(z) = \sqrt{1 - z}$, this becomes 
$$(Q g)(z) = (z \sqrt{1 - z^2} - z^2) g_{e}(z) + (z^2 + z \sqrt{1 - z^2}) g_{e}(z).$$
With respect to the standard basis, $Q$ has the matrix representation, 
$$[Q] = \begin{bmatrix}
\phantom{-} 0 & \phantom{-}  0 & \phantom{-}  0 & \phantom{-}  0 & \phantom{-}  0 & \phantom{-}  0 & \phantom{-} 0 & \phantom{-} 0 & \phantom{-}  0  & \cdots\\
 \phantom{-} 1 & \phantom{-}  0 & \phantom{-}  0 & \phantom{-}  0 & \phantom{-}  0 & \phantom{-}  0 & \phantom{-}  0 & \phantom{-}  0 & \phantom{-}  0 & \cdots\\
 -1 & \phantom{-}  1 & \phantom{-}  0 & \phantom{-}  0 & \phantom{-}  0 & \phantom{-}  0 & \phantom{-}  0 & \phantom{-}  0 & \phantom{-}  0 & \cdots\\
 -\frac{1}{2} & \phantom{-}  1 & \phantom{-}  1 & \phantom{-}  0 & \phantom{-}  0 & \phantom{-}  0 & \phantom{-}  0 & \phantom{-}  0 & \phantom{-}  0 & \cdots\\
 \phantom{-}  0 & -\frac{1}{2} & -1 & \phantom{-}  1 & \phantom{-}  0 & \phantom{-}  0 & \phantom{-}  0 & \phantom{-}  0 & \phantom{-}  0& \cdots \\
 -\frac{1}{8} & \phantom{-}  0 & -\frac{1}{2} & \phantom{-}  1 & \phantom{-}  1 & \phantom{-}  0 & \phantom{-}  0 & \phantom{-}  0 & \phantom{-}  0 & \cdots\\
\phantom{-}   0 & -\frac{1}{8} & \phantom{-} 0 & -\frac{1}{2} & -1 & \phantom{-}  1 & \phantom{-}  0 & \phantom{-}  0 & \phantom{-}  0 & \cdots\\
 -\frac{1}{16} & \phantom{-}  0 & -\frac{1}{8} & \phantom{-}  0 & -\frac{1}{2} & \phantom{-}  1 & \phantom{-}  1 & \phantom{-}  0 & \phantom{-}  0 & \cdots \\
 \phantom{-}  0 & -\frac{1}{16} & \phantom{-}  0 & -\frac{1}{8} & \phantom{-}  0 & -\frac{1}{2} & -1 & \phantom{-}  1 & \phantom{-}  0 & \cdots \\
 \phantom{-} \vdots &  \phantom{-}\vdots & \phantom{-} \vdots &  \phantom{-}\vdots & \phantom{-}  \vdots &  \phantom{-}\vdots &  \phantom{-} \vdots &  \phantom{-}\vdots  &  \phantom{-}\vdots & \ddots
\end{bmatrix}.
$$
\item[(v)] In \eqref{098789u878} take $a \equiv 0$ and $b = c \equiv 1$ to get 
$$(Q g)(z) = z g_{e}(z) + z g_{0}(z) = z g(z)$$ which is just the ``obvious'' square root of $S^2$, namely $S$. 
\item[(vi)] If $a(\D) \subset \D$ (a analytic self map of $\D$), then $1 - z a(z)^2$ is outer and thus $\sqrt{1 -z a(z)^2}$ is a bounded analytic function on $\D$. With $b(z) = c(z) = \sqrt{1 - z a(z)^2}$ (and e.g., $U = I_2$), we can produce a rich class of square roots $Q$ from \eqref{bsgifdgfd} and \eqref{098789u878} as 
\begin{align*}
(Q g)(z) &= \Big(z^2 a(z^2) + z^3 \sqrt{1 - z^2 a(z^2)^2}\Big) g_{e}(z) \\
& \quad \quad  + \Big(\sqrt{1 - z^2 a(z^2)^2} - z^2 a(z^2)\Big) \frac{g_o(z)}{z}.
\end{align*}
\end{enumerate}
\end{Remark}

This brings us to a brief  comment as to when the square root of $S^2$ is a (analytic) Toeplitz operator. Here is a general fact concerning Toeplitz operators. For $\phi \in L^{\infty}(\T)$, define the Toeplitz operator $T_{\phi}$ on $H^2$ by 
$T_{\phi} f = P_{+}(\phi f)$, where $P_{+}$ is the orthogonal projection (the Riesz projection) of $L^2(\T)$ onto $H^2$. See \cite[Ch.~4]{MR3526203} for the basics of Toeplitz operators on $H^2$. 

\begin{Theorem} \label{T:square-toeplitz}
For $\phi \in L^{\infty}(\T)$, the following are equivalent. 
\begin{enumerate}
\item[(i)] There is a Toeplitz operator $T$ such that $T^2 = T_{\phi}$.
\item[(ii)] $\phi = \psi^2$ for some $\psi  \in H^{\infty}$.
\item[(iii)] $\phi = \overline{\psi^2}$ for some $\psi \in H^{\infty}$. 
\end{enumerate}
\end{Theorem}

Below we will use the tensor product $\vec{x} \otimes \vec{y}$ of two vectors in a Hilbert space $\h$. This is the rank-one operator on $\h$ defined by $(\vec{x} \otimes \vec{y})(\vec{z}) = \langle \vec{z}, \vec{y}\rangle \vec{x}$.

\begin{proof}[Proof of Theorem \ref{T:square-toeplitz}]

Let $B=T_\psi$, $\psi\in L^\infty(\mathbb{T})$, denote a (Toeplitz) square root of $T_{\varphi}$, i.e., $T_{\varphi} = (T_\psi)^2$. By a theorem of Brown--Halmos \cite[Ch.~4]{MR3526203},
\[
S^*T_{\psi} S=T_{\psi}
\quad \mbox{and} \quad
S^*T_{\varphi} S= T_{\varphi}.
\]
Recall that $I = SS^*+1\otimes 1$ and thus 
\begin{align*}
T_\varphi &= S^*T_\varphi S\\
&= S^*T_\psi T_\psi S\\
&= S^*T_\psi (SS^*+1\otimes 1)T_\psi S\\
&= S^*T_\psi S\,S^*T_\psi S+S^*T_\psi \,(1\otimes 1)\,T_\psi S\\
&=(T_\psi)^2+(S^*T_\psi1)\otimes (S^*T_\psi^*\,1)\\
&=T_\varphi+(S^*T_\psi1)\otimes (S^*T_\psi^*\,1)\\
&=T_\varphi+(S^*T_\psi1)\otimes (S^*T_{\bar\psi}\,1),
\end{align*}
which implies
\[
(S^*T_\psi1)\otimes (S^*T_{\bar\psi}\,1) = 0.
\]
Therefore, either $S^*T_\psi 1=0$ or $S^*T_{\bar\psi}1=0$. These identities respectively mean $P_+\psi$ or $P_+\bar\psi$ are constant functions. We can rephrase these conditions as $\psi\in \overline{H^\infty}$ or $\psi\in H^\infty$.
However, under any of these two conditions,
\[
T_\varphi= \big( T_\psi \big)^2=T_{\psi^2},
\]
which implies $\varphi=\psi^2$. In the former case, when $\psi\in \overline{H^\infty}$, we may replace $\psi$ by $\bar{\psi}$ so that always $\psi$ is an analytic function and then either $\varphi = \psi^2$ or $\varphi = \overline{\psi}^2$.
\end{proof}

The previous theorem, along with the standard inner-outer factorization of $H^{\infty}$ functions yields the following corollary.

\begin{Corollary}\label{sq:from:anal}
For $\varphi\in H^\infty$, the analytic Toeplitz operator $T_\varphi$ has a square root in the Toeplitz operators if and only if all zeros of $\varphi$ inside the open unit disc $\mathbb{D}$ are of even degrees.
\end{Corollary}

We end this section with the remark that $S^{2n}$ has infinitely many square roots since $S^{2 n}$ is unitarily equivalent to $(S \oplus S)^{(n)}$, and we already know that $S \oplus S$ has infinitely many square roots. However $S^{2 n + 1}$ does not have any square roots. We will discuss these results and some others in a forthcoming paper. 

\section{Square roots of the Volterra operator}

The Volterra operator 
$$(V f)(x) = \int_{0}^{x} f(t)\,dt$$ is a well-known bounded operator on $L^2[0,1]$ with a known square root \cite[p.~81]{MR2003221}
\begin{equation}\label{YYYYYY}
(Y f)(x) = \frac{1}{\sqrt{\pi}} \int_{0}^{x} \frac{f(t)}{\sqrt{x - t}} dt.
\end{equation}
One can prove this using the Laplace transform and convolutions. For the sake of completeness, and since the ideas of the proof will be used in the next section, we give an exposition of the following result of Sarason from \cite{MR208383}. Our presentation will be somewhat different from Sarason.

\begin{Theorem}[Sarason]
The operators $\pm Y$ are the only two square roots of $V$.
\end{Theorem}

 If $\Theta$ is the atomic inner function
$$\Theta(z) = \exp\Big(\frac{z + 1}{z - 1}\Big),$$
a result of Sarason \cite{MR192355} (see also \cite[Ch.~4]{MR2003221}), shows that for $g \in L^{2}[0, 1]$, the function 
$$(J g)(z) = \frac{i \sqrt{2}}{z - 1} \int_{0}^{1} g(t) \Theta(z)^t\,dt, \quad z \in \D,$$ belongs to the model space $K_{\Theta} = H^{2} \cap (\Theta H^2)^{\perp}$ and the operator $J: L^{2}[0, 1] \to \K_{\Theta}$ is unitary. Since $\sigma(V) = \{0\}$, it follows that 
$(I - V)(I + V)^{-1}$ is a bounded operator on $L^2[0, 1]$. The same paper says that 
\begin{equation}\label{6ttWoLk}
J (I - V)(I + V)^{-1} J^{*} = S_{\Theta},
\end{equation}
where $S_{\Theta}$ is  the compression of $S$ to $\K_{\Theta}$, that is $S_{\Theta} = P_{\Theta} S|_{\K_{\Theta}}$, where $P_{\Theta}$ is the orthogonal projection of $H^2$ onto $\K_{\Theta}$. It follows that 
$\sigma(S_{\Theta}) = \{1\}$ and thus
$(I - S_{\Theta}) (I + S_{\Theta})^{-1}$ is a bounded operator  on $\K_{\Theta}$. The compressed shift $S_{\Theta}$ has an $H^{\infty}$ functional calculus in that $\phi(S_{\phi})$ is a well-defined bounded operator on $\K_{\Theta}$ for any $\phi \in H^{\infty}$ \cite[Ch.~9]{MR3526203}.  

For $\psi \in H^{\infty}$, the operator $\psi(S_{\Theta})$ can be written as a truncated Toeplitz operator. Indeed, for any $\psi \in L^{\infty}(\T)$, define the operator $A_{\psi}$ on $\K_{\Theta}$ by $A_{\psi} f = P_{\Theta}(\psi f)$, where $P_{\Theta}$ denotes orthogonal projection of $L^2(\T)$ onto $\K_{\Theta}$ (where we regard $\K_{\Theta}$, via radial boundary values, as a subspace of $L^2(\T)$). Let us record some facts about truncated Toeplitz operators that will be used below. One can find their proofs in \cite{MR3526203} or \cite{MR208383}. 

\begin{Proposition}\label{9Pms 67tyiuwerthbvssS}
Let $\phi \in H^{\infty}$ and $\psi \in L^{\infty}(\T)$. 
\begin{enumerate}
\item[(i)] $A_{z} = S_{\Theta}$. 
\item[(ii)] $A_{\psi} = 0$ if and only if $\psi \in \Theta H^{2} + \overline{\Theta H^{2}}$. 
\item[(iii)] $\phi(S_{\Theta}) = A_{\phi}$. 
\item[(iv)] $\{S_{\Theta}\}' = \{A_{\phi}: \phi \in H^{\infty}\}$. 
\end{enumerate}
\end{Proposition}

Though the operator $(I - S_{\Theta}) (I + S_{\Theta})^{-1}$ is well defined, we need to represent it as a truncated Toeplitz operator with an $H^{\infty}$ symbol. This is accomplished with the following. 

\begin{Proposition}\label{poiurgty887}
If 
$$\phi(z) = \frac{1 - z}{1 + z + \Theta(z)},$$ then $\phi \in H^{\infty}$, is outer, and 
$A_{\phi} = (I - S_{\Theta}) (1 + S_{\Theta})^{-1}.$
\end{Proposition}

\begin{proof}
We first argue that $f(z) = 1 + z + \Theta(z)$ is bounded away from zero on $\D$ (see Figure \ref{graph2})
and thus is an invertible element of $H^{\infty}$. Thus $\phi \in H^{\infty}$. 
\begin{figure}
 \includegraphics[width=.5\textwidth]{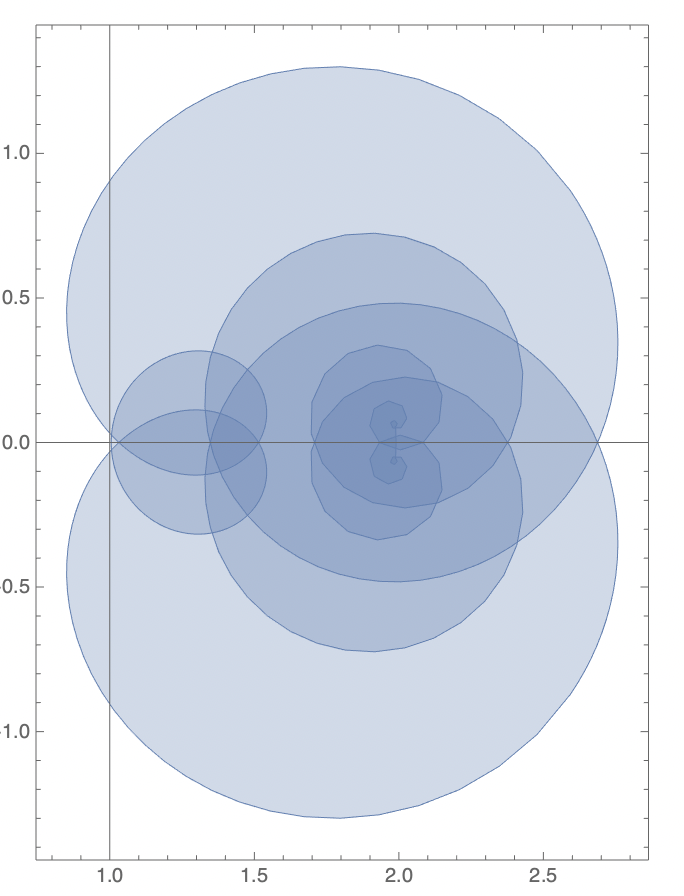}
 \caption{The image of $1 + z + \Theta(z)$ for $z \in \D$. Notice how the closure of this image does not intersect the origin.} 
 \label{graph2}
\end{figure}
Notice that 
$$
\Re f(e^{i\theta}) = 1+\cos\theta + \cos( \cot \theta/2 ).
$$
The function $\cot \theta/2$ is strictly decreasing on $(0,\pi)$ as it moves from $+\infty$ to zero, and at $\theta=\pi/2$ its value is $1$. Hence there is a unique $\theta_0 \in (0,\pi/2)$ such that
$
\cot \theta_0/2 = \pi/2.
$
Fix any $\theta' \in (\theta_0,\pi/2)$ and consider the partition $(0,\pi] = (0,\theta') \cup [\theta',\pi]$. On $(0,\theta')$,
\[
\Re f(e^{i\theta}) = \cos\theta + \bigg(1+\cos( \cot \theta/2 )\bigg) \geq \cos\theta',
\]
and, on $[\theta',\pi]$,
\[
\Re f(e^{i\theta}) = \bigg(1+\cos\theta \bigg) + \cos( \cot \theta/2) \geq \cos( \cot \theta'/2).
\]
Therefore,
$
\Re f(e^{i\theta}) \geq m$ on $\T \setminus \{1\}$, 
where
\[
m = \min\{\cos\theta' , \, \cos( \cot \theta'/2) \} > 0.
\]
By the Poisson integral formula, we conclude that
\[
\Re f(z) = \int_{0}^{2 \pi} \Re f(e^{i \theta})\frac{1 - |z|^2}{|z - e^{i \theta}|^2}\frac{d \theta}{2 \pi} \geq m, \quad z \in \mathbb{D}.
\]
A well known fact says that if $\Re f > 0$ then $f$ is an outer function and thus has no zeros in $\D$ \cite[p.~65]{Garnett}. 

If $\psi(z) = 1 + z + \Theta(z)$, notice that $\phi(z) \psi(z) = 1 - z$ and hence the functional calculus yields $A_{\phi} A_{\psi} = A_{1 - z}$. Proposition \ref{9Pms 67tyiuwerthbvssS} implies that 
$$A_{\psi} = A_{1 + z + \Theta} = A_{1 + z}  + A_{\Theta} = I + S_{\Theta} + 0.$$ Since $A_{1 - z} = I - S_{\Theta}$, it follows that 
$A_{\phi} = (I - S_{\Theta})(I + S_{\Theta})^{-1}$.
\end{proof}

\begin{Corollary}
$V  = J^{*} A_{\phi} J$.
\end{Corollary}

Now let $A \in \mathcal{B}(L^{2}[0, 1])$ such that 
$A^2 = V$.
Lemma \ref{98yuijokl} yields  $A \in \{V\}'$. Since 
$$(I - V)(I + V)^{-1} = I + 2 \sum_{n = 1}^{\infty} (-1)^{n} V^n,$$ then 
$A \in \{(I - V)(I + V)^{-1}\}'$. Note that the series above converges in norm since $V$ is quasinilpotent and thus $\|V^{n}\|^{1/n} \to 0$.  From \eqref{6ttWoLk} we see that $J A J^{*} \in \{S_{\Theta}\}'$. 
Thus $J A J^{*} = A_{\psi}$ for some $\psi \in H^{\infty}$ (Proposition \ref{9Pms 67tyiuwerthbvssS}). Since 
$$A_{\psi}^{2} = (J A J^{*})^2 = J A^2 J^{*} = J V J^{*} = A_{\phi},$$
Proposition \ref{9Pms 67tyiuwerthbvssS} also implies that 
$\psi^2 - \phi \in \Theta H^{2} + \overline{\Theta H^{2}}.$
Since $\psi^2 - \phi$ belongs to $H^{\infty}$ and must also belong to $\Theta H^{\infty} + \overline{\Theta H^{\infty}}$, it follows from  $\overline{\Theta H^{\infty}} \cap H^{\infty} = \C$ that $\psi^2 - \phi \in \Theta H^2$. This will imply that 
$\psi^2 = \phi + \Theta h$ for some $h \in H^{\infty}$.

Recall that $\phi$ is an outer function (and hence is zero free in $\D$) and so there are indeed $\psi \in H^{\infty}$ with $\psi^2 = \phi$. This says that 
$A = J^{*} A_{\psi} J$ for some $\psi \in H^{\infty}$ with $\psi^2 = \phi + \Theta h$ for some $h \in H^{\infty}$. 

On the other hand, if $\psi \in H^{\infty}$  and $h \in H^{\infty}$ with $\psi^2 = \phi + \Theta h$, then the operator $J^{*} A_{\psi} J$ on $L^{2}[0, 1]$ satisfies 
$$
(J^{*} A_{\psi} J)^2  = J^{*} A_{\psi}^{2} J
 = J^{*} (A_{\psi} + A_{\Theta} A_{h}) J
 = J^{*} (A_{\phi} + 0) J
 =  J^{*} A_{\phi} J
 = V.
$$
Note the use of the $H^{\infty}$ functional calculus for the compressed shift $S_{\Theta}$ as well as the fact that $A_{\Theta} = 0$ (Proposition \ref{9Pms 67tyiuwerthbvssS}). This argument is summarized with the following theorem. 

\begin{Theorem}\label{sdfjsdf122445}
For $A \in \mathcal{B}(L^{2}[0, 1])$ the following are equivalent. 
\begin{enumerate}
\item[(i)] $A^2 = V$.
\item[(ii)] $A = J^{*} A_{\psi} J$ for some $\psi \in H^{\infty}$ such that $\psi^2 = \phi + \Theta h$ for some $h \in H^{\infty}$.  
\end{enumerate}
\end{Theorem}

To show there are only two square roots of $V$, we follow a variation of an argument of Sarason \cite{MR208383}. Notice that one square root of $V$ is $J^{*} A_{\sqrt{\phi}} J$. Let us show that the other is $J^{*} A_{-\sqrt{\phi}} J$. If $B$ is another square root of $V$, then $B = J^{*} A_{\psi} J$ where $\psi^2 = \phi + \Theta h$. In other words, $\psi^2 - \phi = \Theta h$. 
Write 
$$\Theta h  = \psi^2 - \phi = (\psi + \sqrt{\phi}) (\psi - \sqrt{\phi})$$
and observe that for some $\gamma_j \geq 0$, the inner functions 
$$q_1(z) = \exp\Big(-\gamma_1 \frac{1 - z}{1 + z}\Big) \quad \mbox{and} \quad q_2(z) = \exp\Big(-\gamma_2 \frac{1 - z}{1 + z}\Big)$$ divide $\psi - \sqrt{\phi}$ and $\psi + \sqrt{\phi}$ respectively. Moreover, choose the largest $\gamma_1, \gamma_2$ such that $q_1$ and $q_2$ are inner divisors of $\psi - \sqrt{\phi}$ and $\psi + \sqrt{\phi}$.
Write 
$$\psi + \sqrt{\phi} = q_1 h_1 \quad \mbox{and} \quad \psi - \sqrt{\phi} = q_2 h_2, \quad h_1, h_1 \in H^{\infty}.$$ It follows that $\sqrt{\phi} = \frac{1}{2} (q_1 h_1 - q_2 h_2).$ Since $\sqrt{\phi}$ is outer, it must be the case that one of $\gamma_1$ or $\gamma_2$ must be zero. If $\gamma_1 > 0$ and $\gamma_2 = 0$, then $\gamma_1 \geq 1$ and it follows that $\psi + \sqrt{\phi}$ is divisible by $\Theta$. An application of Proposition \ref{9Pms 67tyiuwerthbvssS} yields $A_{\psi} = A_{-\sqrt{\phi}}.$

\section{Square root of a compressed shift}

The previous section leads us to a discussion about the square roots of a compressed shift. For {\em any} inner function $u$, there is the compressed shift $S_u = P_{u} S|_{\K_u}$. The proof of Theorem \ref{sdfjsdf122445} implies the following theorem.

\begin{Theorem}\label{bBbsbabbBBB}
For the atomic inner function $\Theta$ and $A \in \mathcal{B}(\K_{\Theta})$, the following are equivalent. 
\begin{enumerate}
\item[(i)] $A^2 = S_{\Theta}$. 
\item[(ii)] $A = A_{\psi}$ for some  $\psi \in H^{\infty}, \psi^2  =  z + \Theta h, h \in H^{\infty}$.
\end{enumerate}
Furthermore, $S_{\Theta}$ has exactly two square roots. 
\end{Theorem}

\begin{proof}
First let us  prove that the set of square roots of $S_u$ is nonempty. 
For this it is enough to  check that 
$z + \Theta(z) (1 - z)^{1/5}$ has no zeros in $\D$ (see Figure \ref{graph1}) and thus has an analytic square root. 
\begin{figure}
 \includegraphics[width=.5\textwidth]{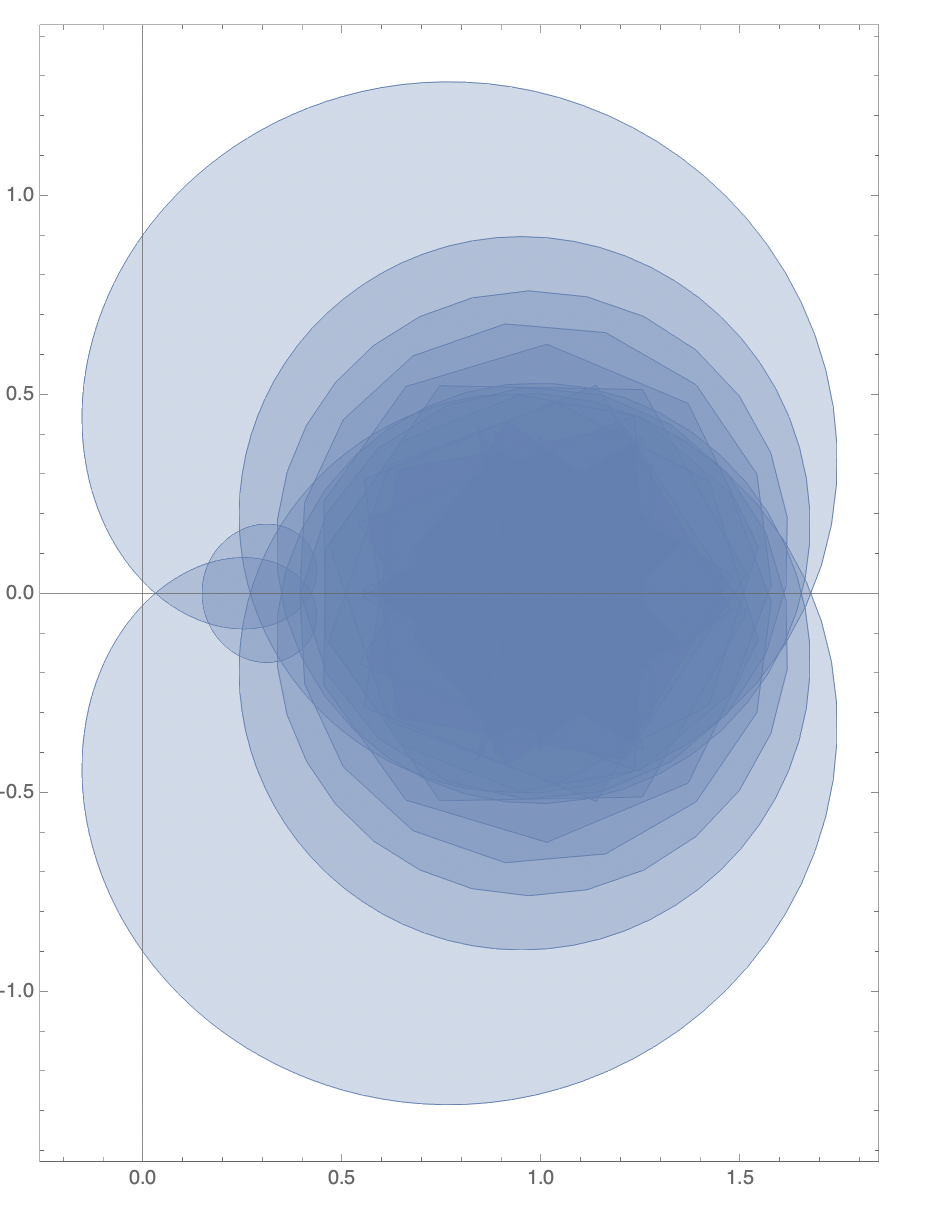}
 \caption{The image of $z + \Theta(z) (1 - z)^{1/5}$ for $z \in \D$. Notice how the closure of this image does not intersect the origin.} 
 \label{graph1}
\end{figure}
The reasoning is similar to the argument in Proposition \ref{poiurgty887}, albeit a bit more complex. In this case
\[
f(z) = z+\Theta(z)(1-z)^{1/5}
\]
and thus
\[
\Re f(e^{i\theta}) = \cos\theta + (2\sin\theta/2)^{1/5} \cos\left(\tfrac{\theta-\pi}{10}-\cot(\theta/2)\right), \quad 0<\theta<\pi.
\]
There is a similar formula for $-\pi<\theta<0$. Then it is enough to observe that
\[
m = \inf_{0<|\theta|\leq\pi} \Re f(e^{i\theta}) > 0.
\]
By the Poisson integral formula, we conclude that 
$\Re f(z) \geq m$ for all $z \in \mathbb{D}$. Thus $f$ is outer and hence has no zeros in $\D$. 

Next we observe that $A_{\sqrt{f}}$ is a square root of $S_{\Theta}$. Now follow the argument used to prove there are only two square roots of the Volterra operator (following Theorem \ref{sdfjsdf122445}) to prove that the other square root of $S_{\Theta}$ is $A_{-\sqrt{f}}$.
\end{proof}

Not every compressed shift has a square root. 

\begin{Proposition}
Suppose $u$ is inner and $u$ has a zero at $z = 0$ of order at least two. Then $S_u$ does not have a square root. 
\end{Proposition}

\begin{proof}
Our earlier discussion shows that the set of square roots of $S_u$ are $\{A_{\psi} : \psi \in H^{\infty}, \psi^2  =  z + u h, h \in H^{\infty}\}$. If $u$ has  a zero of order at least two at $z = 0$, then 
$z + u h = z + z^2 k$ for some $k \in H^{\infty}$ and thus $z + z^2 k(z)$ has a zero of order one at $z = 0$. Thus, there is no $H^{\infty}$ function $\psi$ for which $\psi^2(z) = z + u h$.
\end{proof}

\section{Square roots of  $T_{\cos \theta}$.}

The Toeplitz operator with symbol $\cos \theta$, equivalently 
$$T_{\cos \theta} = \tfrac{1}{2}(S + S^{*}),$$
 is a self-adjoint operator. Therefore, by the spectral theorem for normal operators,  it has a square root. 

A result of Hilbert \cite{MR0056184} (see \cite[Ch.~3]{MR822228} for a modern presentation) shows that if $(u_n)_{n = 0}^{\infty}$ are the Chebyschev polynomials of the second kind \cite{MR0372517},  then the operator $F: L^2(\rho) \to H^2$, where $\rho = \sqrt{1 - x^2}$ on $[-1, 1]$, defined by
\begin{equation}\label{E:def-V}
F u_n = \sqrt{\frac{\pi}{2}} \ z^n, \qquad n \geq 0,
\end{equation}
is unitary and intertwines $M_x$ on $L^2(\rho)$ and $T_{\cos \theta}$. More explicitly,
\[
F M_x = T_{\cos \theta} F.
\]
Thus the matrix representation for $T_{\cos \theta}$ with respect to the orthonormal basis $(z^n)_{n = 0}^{\infty}$  for $H^2$ is $[a_{mn}]_{m, n = 0}^{\infty}$, where
\[
a_{mn} := \langle T_{\cos \theta} z^n, z^m\rangle_{H^2}, \qquad m, n \geq 0,
\]
which is the Toeplitz matrix
$$\begin{bmatrix}
0 & \tfrac{1}{2} & 0 & 0 & 0 & \cdots \\
\tfrac{1}{2} & 0 &  \tfrac{1}{2} & 0 & 0 & \cdots \\
0 &  \tfrac{1}{2} & 0 &  \tfrac{1}{2} & 0 & \cdots \\
0 & 0 &  \tfrac{1}{2} & 0 &  \tfrac{1}{2} & \cdots \\
0 & 0 & 0 &  \tfrac{1}{2} & 0 & \cdots \\[-3pt]
\vdots & \vdots & \vdots & \vdots & \vdots & \ddots
\end{bmatrix}.$$
By means of the unitary operator $F$, we can also write
\[
a_{mn} = \frac{2}{\pi} \int_{-1}^{1} x u_{n}(x) u_{m}(x) \rho(x) dx, \qquad m, n \geq 0.
\]
This observation gives us a path to describe {\em all} of the square roots of $T_{\cos \theta}$. Indeed, if $\phi$ is any measurable function on $[-1, 1]$ for which $\phi(x)^2 = x$ for all $x \in [-1, 1]$, then $M_{\phi}$ satisfies $M_{\phi}^2=M_x$. For example, one can take
\[
\phi(x) = \begin{cases}
\sqrt{x} & \mbox{if } x \geq 0,\\
i \sqrt{-x} & \mbox{if } x < 0.
\end{cases}
\]
Therefore, $F M_{\phi}F^*$ is a square root of $T_{\cos \theta}$.
Of course, there are many other $\phi$ for which $\phi^2 = x$.  The matrix representations of $M_{\phi}$ (with respect to the Chebyschev basis) and $F M_{\phi}F^*$ (with respect to the standard basis for $H^2$) are 
\[
\left[ \frac{2}{\pi}\int_{-1}^{1} \phi(x) u_{n}(x) u_{m}(x) \rho(x) dx\right]_{m, n = 0}^{\infty}.
\]
In fact, these are all the square roots of $A$.

\begin{Theorem}\label{Hilbertertertry}
For $B \in \mathcal{B}(H^2)$ the following are equivalent. 
\begin{enumerate}
\item[(i)] $B^2 = T_{\cos \theta}$. 
\item[(ii)] With respect to the orthonormal basis $(z^n)_{n = 0}^{\infty}$ of $H^2$, the matrix representation of $B$ is 
\begin{equation}\label{E:matrix-rep-sqrtA}
\left[ \frac{2}{\pi}\int_{-1}^{1} \phi(x) u_{n}(x) u_{m}(x) \rho(x) dx\right]_{m, n = 0}^{\infty},
\end{equation}
where $\phi$ is a measurable function on $[-1, 1]$ satisfying $\phi(x)^2 = x$.
\end{enumerate}
\end{Theorem}

\begin{proof}
Let $B$ be any fixed square root of $T_{\cos \theta}$. Since $T_{\cos \theta}$ is unitarily equivalent to $M_x$ on $L^2(\rho)$ via the unitary operator $F$ defined by \eqref{E:def-V}, the operator $F^*BF$ is a square root of $M_x$. By Lemma \ref{98yuijokl},  $F^*BF$ commutes with $M_x$ and thus equal to $M_{\phi}$ for some $\phi \in L^{\infty}(\rho)$ (To see why this is true note, that if $A M_x = M_x A$, then $A p = p A(1)$ for any polynomial $p$. Letting $f \in L^2(\rho)$ and selecting polynomials $p_n \to f$ in norm and, passing to a subsequence, almost everywhere, we see that $A f = (A 1) f$. This shows that $A 1 \in L^{\infty}(\rho)$ and $A$ is the operator multiplication by $A 1$). 
This immediately implies
\[
M_{x} =  (F^*BF)^2 = M_{\phi^2}.
\]
By the uniqueness of the symbol of a multiplication operator, we must have $\phi(x)^2 = x$. The matrix representation of $F^*BF = M_{\phi}$ with respect to the orthonormal basis $(\sqrt{\frac{2}{\pi}} u_n)_{n=0}^{\infty}$ of $L^2(\rho)$ is the same as the matrix representation of $B$ with respect to the orthonormal basis $(z^n)_{n = 0}^{\infty}$ of $H^2$, and is given by \eqref{E:matrix-rep-sqrtA}.
\end{proof}

Notice that all of these square roots are complex symmetric operators, since with respect to the Chebyschev basis, the matrix representation \eqref{E:matrix-rep-sqrtA} is self transpose.

\section{Square roots of the Hilbert matrix}

The square root of the Hilbert matrix 
$$H = \begin{bmatrix}
1 & \frac{1}{2} & \frac{1}{3} & \frac{1}{4} & \cdots & \\[4pt]
\frac{1}{2} & \frac{1}{3} & \frac{1}{4} & \frac{1}{5} & \cdots & \\[4pt]
\frac{1}{3} & \frac{1}{4} & \frac{1}{5} & \frac{1}{6} & \cdots & \\[4pt]
\frac{1}{4} & \frac{1}{5} & \frac{1}{6} & \frac{1}{7} & \cdots &\\
\vdots & \vdots & \vdots & \vdots & \ddots & \\
\end{bmatrix},$$
as an operator on $\ell^2$, involves a similar analysis as with the Toeplitz matrix $T_{\cos \theta}$ from the previous section. But here, the spectral representation theorem of Hilbert is replaced by one of Rosenblum \cite{MR99599}. We outline this analysis here.

The Laguerre polynomials $\{L_{n}(x): n \geq 0\}$ form an orthonormal basis for $L^2((0, \infty), e^{-x} dx)$. A simple integral substitution shows that the map  $(Q f)(x) = e^{-x/2} f(x)$ is unitary from $L^2((0, \infty), e^{-x} dx)$ onto $L^2((0, \infty), dx)$. Thus $$\{Q L_n = e^{-x/2} L_{n}(x): n \geq 0\}$$ is an orthonormal basis for $L^2((0, \infty), dx)$. 

Lebedev \cite{MR0032831, MR0029444} proved that if 
$$K_{\nu}(z) = \int_{0}^{\infty} e^{-z \cosh(t)} \cosh(\nu t)\,dt,$$
the modified Bessel function of the third kind, then 
the operator 
$$(U f)(\tau) = \int_{0}^{\infty} \frac{\sqrt{2 \tau \sinh(\pi \tau)}}{\pi \sqrt{x}} K_{i \tau}(\frac{x}{2}) f(x) dx$$
is a unitary operator from $L^2((0, \infty), dx)$ to itself. Thus 
$$\{w_{n}(x) = U Q L_n: n \geq 0\}$$ 
is an orthonormal basis for $L^2((0, \infty), dx)$. Rosenblum \cite{MR99599} proves that if 
\begin{equation}\label{7y7YnnjnJ}
h(\tau) = \frac{\pi}{\cosh(\pi \tau)},
\end{equation}
then 
$$\langle M_{h} w_m, w_n\rangle_{L^2((0, \infty), dx)} = \frac{1}{n + m + 1}, \quad m, n \geq 0.$$
This last quantity equals $\langle H \vec{e}_{m}, \vec{e}_{n}\rangle_{\ell^2}$ (the entries of the Hilbert matrix). In summary, the linear  transformation $W: \ell^2 \to L^2((0, \infty), dx)$ defined by 
$$W(\{a_n\}_{n \geq 0}) = \sum_{n = 0}^{\infty} a_{n} w_n$$ 
is unitary with $W H W^{*} = M_{h}$.

As in the previous section, if $g \in L^{\infty}((0, \infty), dx)$ with $g^2 = h$, then $M_{g}$ is a square root of $M_h$ and thus $W^{*} M_g W$ is a square root of $H$. Conversely, if $T \in \B(\ell^2)$ with $T^2 = H$, then $W T W^{*}$ is a square root of $M_h$ and hence, as we have seen several times before, $W T W^{*}$ belongs to the commutant of $M_h$.  Since $h$ is a monotone decreasing function on $(0, \infty)$, $h$ is injective and hence by a well-known fact about multiplication operators, $M_h$ is cyclic. Since the commutant of a cyclic multiplication operator is the set of multiplication operators $M_{g}$ on $L^{2}((0, \infty), dx)$ with $g \in L^{\infty}((0, \infty), dx)$ we see as before that 
$T = W^{*} M_g W$, where $g^2 = h$. We therefore arrive at the following theorem. Below, we regard any $T \in \mathcal{B}(\ell^2)$ as an infinite matrix.

\begin{Theorem}\label{HMiiiiI}
For $T \in \B(\ell^2)$ the following are equivalent. 
\begin{enumerate}
\item[(i)] $T^2 = H$.
\item[(ii)] There is a measurable function $g$ on $(0, \infty)$ with $g^2 = h$, where $h$ is the function from \eqref{7y7YnnjnJ}, such that 
$$T = \Big[ \int_{0}^{\infty} g(x) w_m(x) \overline{w_{n}(x)} dx\Big]_{m, n = 0}^{\infty}.$$
\end{enumerate}
\end{Theorem}

\section{Square roots of the Ces\`{a}ro operator}

The Ces\`{a}ro operator $C: H^2 \to H^2$ defined by 
$$(C f)(z) = \frac{1}{z} \int_{0}^{z} \frac{f(\xi)}{1 - \xi} d\xi, \quad z \in \D,$$ is bounded on $H^2$ and a power series computation shows that if $f(z) = \sum_{j = 0}^{\infty} a_{j} z^j \in H^2$, then 
$$(C f)(z) = \sum_{n = 0}^{\infty} \Big(\frac{1}{n + 1} \sum_{j = 0}^{n} a_j\Big) z^n.$$
Some basic facts about $C$ are found in \cite{MR187085}. With resect to the standard orthonormal basis $(z^n)_{n = 0}^{\infty}$ for $H^2$, the matrix representation of $C$ is 
\begin{equation}\label{matrosidjfdsfCesaroo}
\begin{bmatrix}
1 & \0 & \0 & \0 & \0 &  \cdots\\[3pt]
\frac{1}{2} & \frac{1}{2} & \0 & \0 & \0 & \cdots\\[3pt]
\frac{1}{3} & \frac{1}{3} & \frac{1}{3} & \0 & \0 & \cdots\\[3pt]
\frac{1}{4} & \frac{1}{4} & \frac{1}{4} & \frac{1}{4} & \0 & \cdots\\[3pt]
\frac{1}{5} & \frac{1}{5} & \frac{1}{5} & \frac{1}{5} & \frac{1}{5} & \cdots\\[-3pt]
\vdots & \vdots & \vdots & \vdots & \vdots & \ddots
\end{bmatrix}
\end{equation}
which is known as the {\em Ces\`{a}ro matrix}.
Though not quite obvious, $C$ has a square root and in fact, one can write them all down - since  there are only two of them. This is the topic of this section. It is important to note here that by the Conway-Olin functional calculus for subnormal operators \cite{MR399918}, one can prove that $C$ has at least one square root. The purpose here is to show that $C$ has exactly two square roots and to specifically write them down. 

Our path to identify the square roots of $C$ is through subnormal operators and the work of Kriete and Trutt. Along the way to proving that $C$ is subnormal, a paper of Kriete and Trutt \cite{MR281025} shows that for $w \in \D$, the function 
$$v_{w}(z) = (1 - z)^{w/(1 - w)}$$
belongs to $H^2$ and satisfies 
$(I - C^{*}) v_{w} = w v_{w}.$
The space 
$$\mathcal{H} = \{F(z) = \langle f, v_{\overline{z}}\rangle_{H^2}: f \in H^2\}$$ defines a vector space of analytic functions on $\D$ that becomes a Hilbert space, in fact a reproducing kernel Hilbert space, when endowed with the norm $\|F\|_{\mathcal{H}} = \|f\|_{H^2}$. This makes the operator $(U f)(z) = F(z)$ a unitary operator from $H^2$ to $\mathcal{H}$. Furthermore, 
\begin{align*}
(U (I - C) f)(z) & = 
\langle (I - C) f, \phi_{\overline{z}}\rangle_{H^2}\\ & = \langle f, (I - C^{*}) v_{\overline{z}}\rangle_{H^2} \\
& = \langle f, \overline{z} v_{\overline{z}}\rangle_{H^2}\\
& = z \langle f, v_{\overline{z}}\rangle_{H^2}\\
& = z (U f)(z)
\end{align*}
for all $f \in H^2$.
Thus, $U (I - C) = M_{z} U$ on $\mathcal{H}$. In summary, $C$ is unitarily equivalent to $M_{1 - z}$ on $\mathcal{H}$.

Thus if $A$ is a square root of $C$, then, as we have seen with the other operators covered in this paper, $A \in \{C\}'$ and thus $U A  U^{*} \in \{M_{1 - z}\}' = \{M_z\}'$. So now we need to identify $\{M_z\}'$. This will involve the multiplier algebra of $\mathcal{H}$. 

Another paper of Kriete and Trutt \cite{MR350489} argues that each $\phi \in H^{\infty}$ is a multiplier of $\mathcal{H}$ (i.e., $\phi \mathcal{H} \subset \mathcal{H}$). This is significant since for a general Hilbert space of analytic functions (the Dirichlet space for example), not every $H^{\infty}$ function is a multiplier. Furthermore, a standard fact that the multiplier algebra of any reproducing kernel Hilbert space of analytic functions on $\D$ is contained in $H^{\infty}$, along with the observation above, shows that the multiplier algebra of $\mathcal{H}$ is $H^{\infty}$.

The Hilbert space $\mathcal{H}$ also contains the polynomials as a dense set and a standard argument, along with the discussion in the previous paragraph, shows that $\{M_{z}\}' = \{M_{\phi}: \phi \in H^{\infty}\}$. Putting this all together,  it follows that if $A$ is a square root of $C$, then $U A U^{*} = M_{\phi}$ for some $\phi \in H^{\infty}$. But 
$$M_{\phi^2} = M_{\phi}^{2} = (U A U^{*})^2 = U C U^{*} = M_{1 - z}.$$ and thus $\phi^2 = 1 - z$ on $\D$. But since $\phi$ is analytic on $\D$, it must be the case that $\phi(z) = \pm \sqrt{1 - z}$. 
Thus, the Ces\`{a}ro operator has 
$$U^{*} M_{\sqrt{1 - z}} U \quad \mbox{and} \quad U^{*} M_{-\sqrt{1 - z}} U$$
as its only square roots. 

The above formulas for the square roots of $C$ are a bit unsatisfying since they are hidden behind a unitary operator. Our goal in the next two results is to produce a more tangible description of the two square roots of $C$. Note that 
$$\sqrt{1 - z} = 1 - \tfrac{1}{2} z - \tfrac{1}{8} z^2  - \tfrac{1}{16} z^3 - \tfrac{5}{128} z^4 - \cdots  = 1 - \sum_{k = 1}^{\infty} \Big|{\frac{1}{2}  \choose k}\Big| z^k,$$ where the branch of the square root  is taken so that  $\sqrt 1 = 1$. It  is well-known that 
$$\sum_{k = 0}^{\infty} \Big|{\frac{1}{2}  \choose k}\Big|  < \infty.$$

\begin{Theorem}\label{cesarooooo}
The following are equivalent for $A \in \B(H^2)$.
\begin{enumerate}
\item[(i)] $A^2 = C$.
\item[(ii)]
$$A = \pm \Big(I - \tfrac{1}{2} (I - C) -  \tfrac{1}{8} (I - C)^2 - \tfrac{1}{16}  (I - C)^3 + \cdots\Big),$$
where the series above converges in operator norm.
\end{enumerate}
\end{Theorem}

\begin{proof}

From the above discussion, 
$$U^{*} M_{\sqrt{1 - z}} U \quad \mbox{and} \quad U^{*} M_{-\sqrt{1 - z}} U$$ are the only two square roots of $C$. 
 Since $\|M_{z}\| = \|I - C\| = 1$ \cite{MR187085}, the series 
$$I - \tfrac{1}{2} M_z -  \tfrac{1}{8} M_{z}^2 - \tfrac{1}{16}  M_{z}^{3} - \cdots$$ converges in operator norm to 
$M_{\sqrt{1 - z}}$. But since $M_{z}^{k}$ is unitarily equivalent to $(I - C)^{k}$, we get 
\begin{align*}
U^{*} M_{\sqrt{1 - z}} U & = U^{*} \Big(I - \tfrac{1}{2} M_z -  \tfrac{1}{8} M_{z}^2 - \tfrac{1}{16}  M_{z}^{3} - \cdots\Big) U\\
& = I - \tfrac{1}{2} U^{*} M_z U - \tfrac{1}{8} (U^{*} M_z U)^2 - \tfrac{1}{16} (U^{*} M_z U)^3 + \cdots\\
& = I - \tfrac{1}{2} (I - C) -  \tfrac{1}{8} (I - C)^2 - \tfrac{1}{16}  (I - C)^3 + \cdots.
\end{align*}
The other square root of $C$ is computed in a similar way.
\end{proof}

 Using an idea of Hausdorff \cite{MR1544453}, the paper \cite{MR979593} produces all of the lower triangular square roots of the Ces\`{a}ro matrix from \eqref{matrosidjfdsfCesaroo}. That paper considers the Ces\`{a}ro matrix and its resulting square roots as linear transformations on all one-sided sequences (not necessarily $\ell^2$ sequences nor any assumption on the linear transformation being bounded). They show that all of the lower triangular square roots of the Ces\`{a}ro matrix are the matrices $A^{\sigma} = [A^{\sigma}_{ij}]_{i, j = 0}^{\infty}$, where 
\begin{equation}\label{09876tyuhjgyuYUHBGYUIJ}
A^{\sigma}_{ij} =
\begin{cases}
{\displaystyle  {i \choose j} \sum_{\ell = 0}^{i - j} (-1)^{\ell} \sigma(\ell + j + 1) \frac{1}{\sqrt{\ell + j  + 1}} {i - j \choose  \ell}} & i \geq j,\\
 0 & i < j,
 \end{cases}
 \end{equation}
 and $\sigma: \N \to \{-1, 1\}$. One can work out that $A^{\sigma}$ equals
 $$ {\tiny  \begin{bmatrix}
1 & 0 & 0 & 0 & 0 &  \cdots\\[3pt]
1 & -1 & 0 & 0 & 0 & \cdots\\[3pt]
1 & -2 & 1 & 0 & 0 & \cdots\\[3pt]
1 & -3 & 3 & -1 & 0 & \cdots\\[3pt]
1 &-4 & 6 & -4 & 1 & \cdots\\[-3pt]
\vdots & \vdots & \vdots & \vdots & \vdots & \ddots
\end{bmatrix}  \begin{bmatrix}
\pm 1 & 0 &0 & 0 & \cdots\\
0 & \pm\sqrt{\frac{1}{2}}& 0 & 0 & \cdots\\
0 & 0 & \pm \sqrt{\frac{1}{3}} & 0 & \cdots\\
0 & 0 & 0 & \pm \sqrt{\frac{1}{4}} & \cdots\\
\vdots & \vdots & \vdots & \vdots & \ddots
\end{bmatrix}\begin{bmatrix}
1 & 0 & 0 & 0 & 0 &  \cdots\\[3pt]
1 & -1 & 0 & 0 & 0 & \cdots\\[3pt]
1 & -2 & 1 & 0 & 0 & \cdots\\[3pt]
1 & -3 & 3 & -1 & 0 & \cdots\\[3pt]
1 &-4 & 6 & -4 & 1 & \cdots\\[-3pt]
\vdots & \vdots & \vdots & \vdots & \vdots & \ddots
\end{bmatrix}},$$
 where the sign along the diagonal is determined by the function $\sigma$.

  They also conjecture that the choices of $A^{\sigma}$, where $\sigma \equiv 1$ or $\sigma  \equiv  -1$, are the two bounded square roots of the Cesaro matrix (viewed as an operator on $\ell^2$). This next theorem verifies this conjecture (thus answering a question posted by Halmos) and also gives an exact description of the square roots from Theorem \ref{cesarooooo}.
 
 \begin{Theorem}\label{9099887YY66T+}
 The following are equivalent for $A \in \B(H^2)$. 
 \begin{enumerate}
 \item[(i)] $A^2 = C$. 
 \item[(ii)] With respect to the orthonormal  basis $(z^n)_{n = 0}^{\infty}$ for $H^2$, the matrix representation of $A$ is either $[A_{ij}]_{i, j = 0}^{\infty}$ or $-[A_{ij}]_{i, j = 0}^{\infty}$, where
 $$A_{ij} = \begin{cases}
{\displaystyle  {i \choose j} \sum_{\ell = 0}^{i - j} (-1)^{\ell}  \frac{1}{\sqrt{\ell + j  + 1}} {i - j \choose  \ell}} & i \geq j\\
 0 & i < j.
 \end{cases}$$
 \end{enumerate}
 \end{Theorem}
 
 \begin{proof}
 By the above discussion of the results from \cite{MR979593}, all of the lower triangular square roots of the Ces\`{a}ro matrix (as viewed as an operator on the space of all sequences) are of the form $A^{\sigma}$ for some $\sigma: \N \to \{-1, 1\}$. From \eqref{09876tyuhjgyuYUHBGYUIJ}, notice that 
 \begin{equation}\label{okjnbiujhgbvuyhgfvytgf}
 A^{\sigma}_{ii} = \sigma(i + 1) \frac{1}{\sqrt{i + 1}}
 \end{equation}
  and so the choice of $\sigma$ is determined by the entries of $A^{\sigma}$ on its diagonal. If 
 $$A =  \Big(I - \tfrac{1}{2} (I - C) -  \tfrac{1}{8} (I - C)^2 - \tfrac{1}{16}  (I - C)^3 + \cdots\Big),$$
 one of the bounded square roots of the Ces\`{a}ro matrix from Theorem \ref{cesarooooo}, notice that $(I - C)^{k}$ is lower triangular for all $k \geq 0$ and thus so is $A$. We just need to determine which choice of $\sigma$ yields $A^{\sigma} = A$.
 
 The $(n, n)$ entry of $I - C$ is $(1 - \frac{1}{n + 1})$ for $n \geq 0$ and since $I - C$ is lower triangular, it follows that the $(n, n)$ entry of $(I - C)^{k}$ is $(1 - \frac{1}{n + 1})^{k}$. Thus, the $(n, n)$ entry of $A$ is 
 $$1 - \tfrac{1}{2} (1 - \tfrac{1}{n + 1}) - \tfrac{1}{8} (1 - \tfrac{1}{n + 1})^2 - \tfrac{1}{16} (1 - \tfrac{1}{n + 1})^3 - \cdots.$$
 But the above is just the Taylor series of $\sqrt{1 - z}$ evaluated at $z =1 -  \frac{1}{n + 1}$ and this turns out to be $\sqrt{\frac{1}{n +1}}$.  By \eqref{okjnbiujhgbvuyhgfvytgf}, this  corresponds to $A^{\sigma}$ with $\sigma \equiv 1$. 
 
 When 
 $$A = - \Big(I - \tfrac{1}{2} (I - C) -  \tfrac{1}{8} (I - C)^2 - \tfrac{1}{16}  (I - C)^3 + \cdots\Big),$$
 a similar analysis shows that corresponds to $A^{\sigma}$ with $\sigma \equiv -1$. 
   \end{proof}
   
 Thus the only two bounded square roots of the Ces\`{a}ro (matrix) operator are 
 $${\tiny  \begin{bmatrix}
1 & 0 & 0 & 0 & 0 &  \cdots\\[3pt]
1 & -1 & 0 & 0 & 0 & \cdots\\[3pt]
1 & -2 & 1 & 0 & 0 & \cdots\\[3pt]
1 & -3 & 3 & -1 & 0 & \cdots\\[3pt]
1 &-4 & 6 & -4 & 1 & \cdots\\[-3pt]
\vdots & \vdots & \vdots & \vdots & \vdots & \ddots
\end{bmatrix}  \begin{bmatrix}
1 & 0 &0 & 0 & \cdots\\
0 & \sqrt{\frac{1}{2}}& 0 & 0 & \cdots\\
0 & 0 & \sqrt{\frac{1}{3}} & 0 & \cdots\\
0 & 0 & 0 & \sqrt{\frac{1}{4}} & \cdots\\
\vdots & \vdots & \vdots & \vdots & \ddots
\end{bmatrix}\begin{bmatrix}
1 & 0 & 0 & 0 & 0 &  \cdots\\[3pt]
1 & -1 & 0 & 0 & 0 & \cdots\\[3pt]
1 & -2 & 1 & 0 & 0 & \cdots\\[3pt]
1 & -3 & 3 & -1 & 0 & \cdots\\[3pt]
1 &-4 & 6 & -4 & 1 & \cdots\\[-3pt]
\vdots & \vdots & \vdots & \vdots & \vdots & \ddots
\end{bmatrix}} $$ and 
$${\tiny  \begin{bmatrix}
1 & 0 & 0 & 0 & 0 &  \cdots\\[3pt]
1 & -1 & 0 & 0 & 0 & \cdots\\[3pt]
1 & -2 & 1 & 0 & 0 & \cdots\\[3pt]
1 & -3 & 3 & -1 & 0 & \cdots\\[3pt]
1 &-4 & 6 & -4 & 1 & \cdots\\[-3pt]
\vdots & \vdots & \vdots & \vdots & \vdots & \ddots
\end{bmatrix}  \begin{bmatrix}
-1 & 0 &0 & 0 & \cdots\\
0 & -\sqrt{\frac{1}{2}}& 0 & 0 & \cdots\\
0 & 0 & -\sqrt{\frac{1}{3}} & 0 & \cdots\\
0 & 0 & 0 & -\sqrt{\frac{1}{4}} & \cdots\\
\vdots & \vdots & \vdots & \vdots & \ddots
\end{bmatrix}\begin{bmatrix}
1 & 0 & 0 & 0 & 0 &  \cdots\\[3pt]
1 & -1 & 0 & 0 & 0 & \cdots\\[3pt]
1 & -2 & 1 & 0 & 0 & \cdots\\[3pt]
1 & -3 & 3 & -1 & 0 & \cdots\\[3pt]
1 &-4 & 6 & -4 & 1 & \cdots\\[-3pt]
\vdots & \vdots & \vdots & \vdots & \vdots & \ddots
\end{bmatrix}}.$$

\begin{Remark}

It is important to now that any other option of sign along the main diagonal of the middle matrix will yield an unbounded operator on $\ell^2$. For example, 
 $${\tiny  \begin{bmatrix}
1 & 0 & 0 & 0 & 0 &  \cdots\\[3pt]
1 & -1 & 0 & 0 & 0 & \cdots\\[3pt]
1 & -2 & 1 & 0 & 0 & \cdots\\[3pt]
1 & -3 & 3 & -1 & 0 & \cdots\\[3pt]
1 &-4 & 6 & -4 & 1 & \cdots\\[-3pt]
\vdots & \vdots & \vdots & \vdots & \vdots & \ddots
\end{bmatrix}  \begin{bmatrix}
-1 & 0 &0 & 0 & \cdots\\
0 & \sqrt{\frac{1}{2}}& 0 & 0 & \cdots\\
0 & 0 & \sqrt{\frac{1}{3}} & 0 & \cdots\\
0 & 0 & 0 & \sqrt{\frac{1}{4}} & \cdots\\
\vdots & \vdots & \vdots & \vdots & \ddots
\end{bmatrix}\begin{bmatrix}
-1 & 0 & 0 & 0 & 0 &  \cdots\\[3pt]
1 & -1 & 0 & 0 & 0 & \cdots\\[3pt]
1 & -2 & 1 & 0 & 0 & \cdots\\[3pt]
1 & -3 & 3 & -1 & 0 & \cdots\\[3pt]
1 &-4 & 6 & -4 & 1 & \cdots\\[-3pt]
\vdots & \vdots & \vdots & \vdots & \vdots & \ddots
\end{bmatrix}} $$
(notice the minus sign in the first entry of the diagonal matrix and all the other entries are positive) can be written as 
 $${\tiny  \begin{bmatrix}
1 & 0 & 0 & 0 & 0 &  \cdots\\[3pt]
1 & -1 & 0 & 0 & 0 & \cdots\\[3pt]
1 & -2 & 1 & 0 & 0 & \cdots\\[3pt]
1 & -3 & 3 & -1 & 0 & \cdots\\[3pt]
1 &-4 & 6 & -4 & 1 & \cdots\\[-3pt]
\vdots & \vdots & \vdots & \vdots & \vdots & \ddots
\end{bmatrix}  \begin{bmatrix}
1 - 2 & 0 &0 & 0 & \cdots\\
0 & \sqrt{\frac{1}{2}}& 0 & 0 & \cdots\\
0 & 0 & \sqrt{\frac{1}{3}} & 0 & \cdots\\
0 & 0 & 0 & \sqrt{\frac{1}{4}} & \cdots\\
\vdots & \vdots & \vdots & \vdots & \ddots
\end{bmatrix}\begin{bmatrix}
1 & 0 & 0 & 0 & 0 &  \cdots\\[3pt]
1 & -1 & 0 & 0 & 0 & \cdots\\[3pt]
1 & -2 & 1 & 0 & 0 & \cdots\\[3pt]
1 & -3 & 3 & -1 & 0 & \cdots\\[3pt]
1 &-4 & 6 & -4 & 1 & \cdots\\[-3pt]
\vdots & \vdots & \vdots & \vdots & \vdots & \ddots
\end{bmatrix}} $$
which is equal to 
 $${\tiny  \begin{bmatrix}
1 & 0 & 0 & 0 & 0 &  \cdots\\[3pt]
1 & -1 & 0 & 0 & 0 & \cdots\\[3pt]
1 & -2 & 1 & 0 & 0 & \cdots\\[3pt]
1 & -3 & 3 & -1 & 0 & \cdots\\[3pt]
1 &-4 & 6 & -4 & 1 & \cdots\\[-3pt]
\vdots & \vdots & \vdots & \vdots & \vdots & \ddots
\end{bmatrix}  \begin{bmatrix}
1 & 0 &0 & 0 & \cdots\\
0 & \sqrt{\frac{1}{2}}& 0 & 0 & \cdots\\
0 & 0 & \sqrt{\frac{1}{3}} & 0 & \cdots\\
0 & 0 & 0 & \sqrt{\frac{1}{4}} & \cdots\\
\vdots & \vdots & \vdots & \vdots & \ddots
\end{bmatrix}\begin{bmatrix}
1 & 0 & 0 & 0 & 0 &  \cdots\\[3pt]
1 & -1 & 0 & 0 & 0 & \cdots\\[3pt]
1 & -2 & 1 & 0 & 0 & \cdots\\[3pt]
1 & -3 & 3 & -1 & 0 & \cdots\\[3pt]
1 &-4 & 6 & -4 & 1 & \cdots\\[-3pt]
\vdots & \vdots & \vdots & \vdots & \vdots & \ddots
\end{bmatrix}}$$
 $$- 2{\tiny  \begin{bmatrix}
1 & 0 & 0 & 0 & 0 &  \cdots\\[3pt]
1 & -1 & 0 & 0 & 0 & \cdots\\[3pt]
1 & -2 & 1 & 0 & 0 & \cdots\\[3pt]
1 & -3 & 3 & -1 & 0 & \cdots\\[3pt]
1 &-4 & 6 & -4 & 1 & \cdots\\[-3pt]
\vdots & \vdots & \vdots & \vdots & \vdots & \ddots
\end{bmatrix}  \begin{bmatrix}
1 & 0 &0 & 0 & \cdots\\
0 & 0 & 0 & 0 & \cdots\\
0 & 0 & 0 & 0 & \cdots\\
0 & 0 & 0 & 0 & \cdots\\
\vdots & \vdots & \vdots & \vdots & \ddots
\end{bmatrix}\begin{bmatrix}
1 & 0 & 0 & 0 & 0 &  \cdots\\[3pt]
1 & -1 & 0 & 0 & 0 & \cdots\\[3pt]
1 & -2 & 1 & 0 & 0 & \cdots\\[3pt]
1 & -3 & 3 & -1 & 0 & \cdots\\[3pt]
1 &-4 & 6 & -4 & 1 & \cdots\\[-3pt]
\vdots & \vdots & \vdots & \vdots & \vdots & \ddots
\end{bmatrix}}.$$
The first matrix in the sum above is one of the bounded square roots of the Ces\`{a}ro operator while the second matrix in the sum turns out to be 
 $${\tiny  \begin{bmatrix}
1 & 0 & 0 & 0 & 0 &  \cdots\\[3pt]
1 & 0 & 0 & 0 & 0 & \cdots\\[3pt]
1 & 0 & 0 & 0 & 0 & \cdots\\[3pt]
1 & 0 & 0 & 0 & 0 & \cdots\\[3pt]
1 & 0 & 0 & 0 & 0 & \cdots\\[-3pt]
\vdots & \vdots & \vdots & \vdots & \vdots & \ddots
\end{bmatrix}}$$
which is clearly an unbounded operator on $\ell^2$. 
\end{Remark}
\bibliographystyle{plain}

\bibliography{references}

\end{document}